\documentclass[11 pt]{amsart}

\usepackage[margin=2.5cm]{geometry}

\usepackage{enumerate,tikz-cd,mathtools,amssymb}

\definecolor{cite}{RGB}{44,123,182}
\definecolor{ref}{RGB}{215,25,28}
\usepackage[pdftex,
              pdfauthor={Giulia Gugiatti},
              pdftitle={Title},
              pdfsubject={},
              pdfkeywords={},
              colorlinks=true,
              linkcolor=ref 
              ]{hyperref}

\usepackage{verbatim}
\usepackage[all,cmtip]{xy}
\usepackage{libertine}
\usepackage{adjustbox}
\usepackage{multirow}
\usepackage{makecell}
\newtheorem{thm}{Theorem}
\newtheorem{pro}[thm]{Proposition}
\newtheorem{lem}[thm]{Lemma}

\theoremstyle{definition}
\newtheorem{dfn}[thm]{Definition}

\newtheorem{rem}[thm]{Remark}

\DeclareMathOperator{\Hom}{Hom}

\DeclareMathOperator{\mult}{mult}

\DeclareMathOperator{\Spec}{Spec}

\DeclareMathOperator{\gr}{gr}

\newcommand{\oo}{\mathcal{O}}
\newcommand{\QQ}{\mathbb{Q}}
\newcommand{\HH}{\mathbb{H}}

\newcommand{\PP}{\mathbb{P}}
\newcommand{\ZZ}{\mathbb{Z}}

\newcommand{\CC}{\mathbb{C}}

\newcommand{\LL}{\mathbb{L}}
\newcommand{\TT}{\mathbb{T}}

\begin{document}
\title{A Prym hypergeometric}

\author[A.Corti]{Alessio Corti}
\address{AC: Department of Mathematics\\ Imperial College London\\ London SW7
  2AZ UK} 
\email{a.corti@imperial.ac.uk}

\author[G. Gugiatti]{Giulia Gugiatti}
\address{GG: Math Section \\ ICTP \\ Leonardo Da Vinci Building,
  Strada Costiera 11\\
  34151 Trieste, Italy} 
\email{ggugiatt@ictp.it}

\author[F. Rodriguez Villegas]{Fernando Rodriguez Villegas}
\address{FRV: Math Section \\ ICTP \\ Leonardo Da Vinci Building,
  Strada Costiera 11\\
  34151 Trieste, Italy} 
\email{villegas@ictp.it}

\subjclass[2020]{}

\begin{abstract} We study a hypergeometric local system that
  arises from the quantum Chen--Ruan cohomology of a family of
  weighted del Pezzo hypersurfaces. We
  prove that it is the anti-invariant variation of a pencil of
  genus-$7$ curves with respect to an involution having $4$ fixed points.   
\end{abstract}

\maketitle
\setcounter{tocdepth}{1}
\tableofcontents{}

 
\section{Introduction}
\label{sec:introduction}
Consider the hypergeometric function	 defined by the complex power series:
\begin{equation} \label{eq:hgm-function}
F(\alpha)	=\sum _{j=0}^{\infty} \frac{(18j)! j! }{(2j)! (3j)! (5j)! (9j)!} \alpha^j
\end{equation} 
This function satisfies an irreducible order--$8$ hypergeometric operator $H$ on $\CC^\times$ singular at $\alpha_0=\frac{5^5}{3^{15} 2^{16}}$--- written out explicitly in Equation~\eqref{eq:irreducible} below.
The local system $\HH$ of its solutions is a complex irreducible local system of rank 8 on $U=\CC^\times \setminus \{\alpha_0 \}$. By~\cite[Theorem 5.4.4]{NK} $\HH$ supports a canonical rational variation of (pure) Hodge structures (VHS) of weight one~\cite[Section 1.2]{CG}. 

Our main result is the following.


\begin{thm}\label{thm:main}
Consider the pencil of genus-7 curves $\widetilde{N}_\alpha$ defined by the equation:
\begin{equation} \label{eq:Ntilde} \left(4  y^3 -\alpha x_0 x_1^2 y^2 +	4\alpha^3 x_1^6 y  +\alpha x_0^9  =0
  \right) \subset \PP(1,1,3)
\end{equation}
where $x_0,x_1,y$ are homogeneous coordinates of weights $1,1,3$, and let   
 $P_\alpha \colon \widetilde{N}_\alpha \to N_\alpha$ be the quotient by the involution $x_1 \mapsto -x_1$.  
Then, for all $\alpha \in \CC^\times \setminus \{\alpha_0 \}$:
\begin{equation}
  \HH_\alpha = H^1(\widetilde{N}_\alpha, \ZZ)^-\otimes_\ZZ \QQ
\end{equation} where $H^1(\widetilde{N}_\alpha, \ZZ)^-$ are the antiinvariants under the involution. In particular this endows $\mathbb{H}$ with an integral structure.
\end{thm}

\begin{rem}
  The curves $\widetilde{N}_\alpha, N_\alpha$ are nonsingular for all $\alpha \in \CC^\times \setminus \{\alpha_0\}$, and the involution $x_1 \mapsto -x_1$ on $\widetilde{N}_\alpha$ has four fixed points, hence $H^1(\widetilde{N}_\alpha, \QQ)^-$ has rank $8$ --- see Remark~\ref{rem:map-P} .
\end{rem}

Theorem~\ref{thm:main} follows from a general result on conic bundles over surfaces 
which is of independent interest (we set out our terminology for conic bundles in Section~\ref{sec:conic-notions}). 

\begin{thm} \label{thm:sheaf-conic-bdl} Let $\pi \colon  X \to S$ be a conic bundle over a surface $S$ with ramification data  $p \colon \widetilde{\Delta} \to \Delta$, and let $\mathbb{W}$ be the subsheaf of $p_\star \ZZ_{\widetilde{\Delta}}$ of anti-invariants with respect to the involution on $\widetilde{\Delta}$. Let $i \colon \Delta \hookrightarrow S$ be the inclusion. 

There is a short exact sequence of mixed sheaves on $S$:
\begin{equation} \label{eq:R^2-and-W}
0 \to i_\star \mathbb{W} \to R^2 \pi_\star \ZZ_X(1) \to \ZZ_S \to 0	
\end{equation} 
\end{thm}
We prove Theorem \ref{thm:sheaf-conic-bdl} in Section \ref{sec:conic-proof}.

\begin{rem} If $ X \to S$ is a conic bundle over a rational projective surface $S$, the works \cite{Bea1, Bel, Bea2} describe $H^3(X, \ZZ)$ in terms of the ramification data $ \widetilde{\Delta} \to \Delta$.  Our Theorem \ref{thm:sheaf-conic-bdl} is a sheaf-theoretic generalisation of these results which applies to conic bundles $ X \to S$ where $X$ is not necessarily proper.  \end{rem}

\subsection{Motivation and context} 
Our motivation to study the hypergeometric local system $\HH$ is twofold.
On the one hand $\HH$ supports a VHS of weight one,
hence it is natural to ask if $\HH$ is related to the variation of $H^1$ of a pencil of curves. 
Our Theorem \ref{thm:main} proves that this is indeed the case.

On the other hand, 
the function $F$ in Equation \eqref{eq:hgm-function} is, up to a shift in $\alpha$, a specialisation of the regularised quantum period of the family of log del Pezzo surfaces $X_{18} \subset \PP(2,3,5,9)$. This is an almost immediate consequence of \cite[Proposition 37]{ACGG}. Thus our Theorem \ref{thm:main} implies that the pencil of curves $\widetilde{N}_\alpha$  is a Landau--Ginzburg mirror of this family.

The family $X=X_{18} \subset \PP(2,3,5,9)$ appears as one of the sporadic cases of the classification of anticanonical quasismooth and wellformed log del Pezzo surfaces, by Johnson and Koll\'ar \cite[Theorem 8]{JK}. The mirror construction we present here is the only one that we know for the surfaces $X$. Indeed, since the anticanonical linear system of $X$ is empty, none of the standard methods apply, see \cite[Remark 2.7]{ACGG}.
\smallskip

 We refer to \cite[Section 2--9]{R-V} for an introduction to the geometry of hypergeometric motives. We refer to \cite[Sections~1,2]{ACGG} for our notion of mirror symmetry for Fano varieties with quotient singularities. 


\subsection{Hypergeometric GKZ systems}
In Section \ref{sec:3folds} we consider an explicit affine threefold $Z_\alpha$ depending on the paramater $\alpha \in \CC^\times$, see for details Equation \eqref{eq:-generalLaurent-poly},
and we state that the VHS supported by $\HH$ is the variation $\gr^W_3 H^3_c(Z_\alpha, \QQ)(1)$ over $\alpha \neq \alpha_0$  (Proposition \ref{pro:our-claim}). 

This identity of VHS follows from a general result by Stienstra \cite{Stienstra}, which shows that certain $\mathcal{D}$-modules associated to Gel'fand--Kapranov--Zelevinsky hypergeometric systems (GKZ systems) are related to relative cohomology modules $H^{\bullet}(\TT, Z_{v}, \QQ)$, where $Z_{v}=(f_v=0) \subset \TT$ is the hypersurface cut out of the torus $\TT$ by a parametric Laurent polynomial $f_v$ built out of the GKZ data.  The hypergeometric $\mathcal{D}$-module associated to our operator $H$ arises from the restriction of such a $\mathcal{D}$-module to a hyperplane in the parameter space, and the pencil $Z_\alpha$ is obtained from the restriction of the corresponding parametric family $Z_{v}$ to the same hyperplane.

\subsection{Sketch of the proof of Theorem \ref{thm:main}} The starting point for the proof of Theorem \ref{thm:main} is the pencil of threefolds $Z_\alpha$ of Section \ref{sec:3folds} where $\; \HH=\gr^W_3 H^3_c(Z_\alpha, \QQ)(1)$ over $\alpha \neq \alpha_0$. 

In Section~\ref{sec:conic-structure}, we construct a partial compactification $Z_\alpha \subset X_\alpha \subset (\CC^\times)^2 \times \PP^2$ such that the projection $\pi \colon  X_\alpha  \to (\CC^\times)^2$ to the first factor is a conic bundle.

Once we have the conic bundle, the rest of the proof follows from an application of Theorem \ref{thm:sheaf-conic-bdl} to the case at hand, together with manipulations in mixed Hodge theory. 
The quotient map $P \colon \widetilde{N}_\alpha \to  {N}_\alpha$ 
of Theorem \ref{thm:main} 
arises from the ramification data of the conic bundle $\pi \colon  X_\alpha  \to (\CC^\times)^2$. We believe that Theorem \ref{thm:sheaf-conic-bdl} is the easiest way to study the Hodge structure $\gr^W_3 H^3_c(X_\alpha, \QQ)$. 

\smallskip

The map $P  \colon \widetilde{N}_\alpha \to  {N}_\alpha$
is a double cover of a genus $3$-curve with  branch divisor $B_\alpha$ of degree $4$, see Remark \ref{rem:map-P}. By~\cite[Theorem 9.14]{MR1360615} the Prym map $\mathcal{R}_{3,4} \to \mathcal{A}_4^\delta$, $\delta=(1,2,2,2)$, is a dominant morphism generically of degree $3$, but note that the construction \cite[9.1]{MR1360615} does not apply here since the linear system $|B_\alpha|$ has a base point, see Remark~\ref{rem:map-P}.


\subsection{Structure of the paper} Section \ref{sec:initial-set-up} contains the initial setup. In Section \ref{sec:3folds} we  construct the pencil of threefolds $Z_\alpha$ and we state Proposition \ref{pro:our-claim}; in Section \ref{sec:conic-structure} we build the partial compactification $Z_\alpha \subset X_\alpha$. In Section \ref{sec:conic}  we prove Theorem \ref{thm:sheaf-conic-bdl}. In Section \ref{sec:proof} we prove Theorem \ref{thm:main}. In the Appendix we prove Proposition \ref{pro:our-claim}. 

\subsection{Acknowledgements} We thank Fabrizio Catanese, Tom Coates, Juan Carlos Naranjo, Gian Pietro Pirola, and Alessandro Verra for helpful discussions. It is a pleasure and an honour to have the opportunity to contribute to the Collino Volume. AC is partially supported by EPSRC Program 
Grant EP/N03189X/1.



\section{Initial setup} 
\label{sec:initial-set-up}
In this section we introduce the main objects involved in the proof of Theorem \ref{thm:main}.

\medskip

\subsection{A family of threefolds} \label{sec:3folds} Consider the
series $F$ in Equation \eqref{eq:hgm-function}. Following
\cite{FRV-DUKE} and \cite[Section 2]{R-V}, we will encode the data defining the coefficients of
$F$ in the gamma list $\gamma=(-18,-1,2,3,5,9)$.  The function $F$ is
a solution to the order-$8$ irreducible hypergeometric operator on
$\CC^\times$:
\begin{equation} \label{eq:irreducible}
	H=\alpha_0 \cdot  D\left(D-\frac{1}{3}\right) \left(D-\frac{2}{3}\right) \prod_{n=0}^{4} \left(D-\frac{n}{5}\right)  - \alpha\prod_{\substack{n=0\\  n \neq 4} }^{8}\left(D+\frac{2n+1}{18}\right)
\end{equation}
where $\alpha_0$ is as in the Introduction and $D=\alpha \frac{d}{d \alpha}$. The local system $\HH$ of the Introduction is the local
system of solutions of $H$.
\begin{rem}
The connection between  the hypergeometric parameters in \eqref{eq:irreducible} and the list $\gamma$ is given by the monodromy representation  
$\rho \colon \pi_1(\PP^1 \setminus \{0, \alpha_0, \infty \}) \to \mathrm{GL}(\CC^8)$ of $H$. 
Let $\gamma_0, \gamma_{\alpha_0}, \gamma_{\infty}$ be three generators of \mbox{$\pi_1(\PP^1 \setminus \{0, \alpha_0, \infty \})$} such that $\gamma_\infty \gamma_{\alpha_0} \gamma_0=1$, and let $q_\infty$ and $q_0$ be the characteristic polynomials of $\rho(\gamma_\infty)$ and $\rho(\gamma_0)^{-1}$. Then by \cite[Section 3]{BH} 
\[ 
\frac{q_\infty}{q_0}(\alpha)=  \frac{\phi_{18} \phi_1 }{\phi_2 \phi_3 \phi_5 \phi_9} (\alpha)=
\frac{ (\alpha^{18}-1)(\alpha -1) }{ (\alpha^2-1)(\alpha^3-1) (\alpha^5-1) (\alpha^9-1) }
\] 
where $\phi_n$ denotes the $n$-th cyclotomic polynomial.
\end{rem}
\medskip

Now consider vectors $m_1, \dots, m_6 \in \ZZ^4$ whose affine span is primitive and such that $\gamma$ spans their affine relations, that is, $\sum_{i=1}^{6} \gamma_i \; m_i=0$.
Let $u_1, \dots, u_4$ be coordinates on  the torus $\TT=\mathrm{Hom}(\ZZ^4, \CC^\times) \simeq (\CC^\times)^4$. Pick integers $k_1, \dots, k_6$ such that $\sum_{i=1}^{6} k_i \gamma_i=1$, and consider the Laurent polynomial
\begin{equation}\label{eq:-generalLaurent-poly}
f_\alpha(u)=\sum_{i=1}^{6}  (-\alpha)^{k_i} {u}^{m_i}
\end{equation} where $u=(u_1, \dots, u_4)$, ${u}^{m}\coloneqq u_1^{m_1}u_2^{m_2} u_3^{m_3} u_4^{m_4}$ for $m=(m_1, \dots, m_4) \in \ZZ^4$ and $\alpha$ is a parameter in $\CC^\times$. 
Denote by $Z_\alpha$ the threefold $Z_\alpha =(f_\alpha=0) \subset \TT$.
It is easy to check that $Z_\alpha$ is singular if and only if $\alpha=\alpha_0$. The threefold $Z_\alpha$ is refered to as the \emph{toric model} associated to $\gamma$ in \cite[Section 3]{R-V}.
\begin{rem} \label{rem:unique}
The vectors $m_i$ 
are uniquely determined up to invertible affine linear transformations of $\ZZ^4$. Therefore, the associated hypersurface $Z_\alpha \subset \TT^4$ is uniquely determined up to isomorphism.
\end{rem}

In the Appendix we prove that:
\begin{pro} \label{pro:our-claim}
The local system $\HH$ is the variation $	\gr_3^W H^3_c(Z_\alpha, \QQ)(1)	$, $\alpha \in \CC^\times \setminus \{\alpha_0\}$. 
\end{pro}

\subsection{A conic bundle structure} \label{sec:conic-structure}
In the rest of the paper we will choose the vectors $m_i$ as the columns of the matrix:
\begin{equation} \label{eq:our-monomials}
	\left( \begin{array}{cccccc}
1 & 0 & 0 &  3 & 0 & 1 \\
1 & 0 & 0 &  1 & 3 & 0\\
1 & 0 & 0 &  0 & 0 & 2 \\ 
0 & 2 & 1 &  0 & 0 & 0
\end{array} \right)
\end{equation}
and we will set $(k_1, \dots, k_6)=(0,-1,0,0,0,0)$. Then the threefold $Z_\alpha \subset \TT$ is defined by the equation:
\begin{equation} \label{eq:our-choice-mi}
u_1u_2u_3+ u_4+u_1^3u_2+u_2^3+u_1u_3^2-\frac{1}{\alpha}u_4^2=0	
\end{equation}
It is clear that $Z_\alpha$ is the intersection of the threefold $X_\alpha \subset (\CC^\times)^2 \times \PP^2$ defined by the equation:
\begin{equation} \label{eq:conic-bdl-1}
\left(u_1^3u_2+u_2^3\right) x_0^2 +u_1u_2 \; x_0x_1 + u_1 \; x_1^2 +x_0x_2 -\frac{1}{\alpha} x_2^2=0	
\end{equation}
with the torus $(\CC^\times)^4 \subset (\CC^\times)^2  \times \PP^2$, where $u_1, u_2$ are coordinates on $(\CC^\times)^2$, $x_0, x_1, x_2$ are homogeneous coordinates on $\PP^2$, and, for $x_0 \neq 0$, $u_3=x_1/x_0$, $u_4=x_2/x_0$.
The projection $\pi \colon X_\alpha \to (\CC^\times)^2$ to the first factor is a conic bundle over $(\CC^\times)^2$. 

\begin{rem} Write $\{e_i\}$, $i=1,2,3,4$, for the standard basis of $\ZZ^4$. With the choice $m_1=\underline{0}$,  $m_2=(2,3,5,9)$, $(m_3,m_4,m_5, m_6)=(e_1, e_2, e_3, e_4)$, and $(k_1, \dots, k_6)=(0,1,0,0,0,0)$, one obtains the equation:  \begin{equation} \label{eq:bcm}
1+u_1+u_2+u_3+u_4 -\frac{1}{\alpha}u_1^2u_2^3u_3^5u_4^9=0	
\end{equation} 
This is the same equation associated to the list $\gamma$ in~\cite[Equation (1.1)]{BCM}.
After the change of coordinates 
\[ u_1 \mapsto \frac{u_4}{u_1u_2u_3} \quad u_2 \mapsto \frac{u_1^2}{u_3} \quad u_3 \mapsto  \frac{u_2^2}{u_1u_3}\quad u_4 \mapsto  \frac{u_3}{u_2}
\] Equation \eqref{eq:bcm}, multiplied by $u_1u_2u_3$, becomes
Equation \eqref{eq:our-choice-mi}.  The point of our choice
\eqref{eq:our-monomials} is to make the conic bundle structure
\eqref{eq:conic-bdl-1} manifest.
\end{rem}



\section{Conic bundles over surfaces}  
\label{sec:conic}
\subsection{Basic notions} \label{sec:conic-notions}

We recall some basic notions on conic bundles, see~\cite{YP} for an
extensive survey. 

\begin{dfn} A \emph{conic bundle} is a projective morphism $\pi \colon X\to
  S$ where $X$ is a nonsingular \mbox{3-fold}, $S$ is a surface,
  $\pi_\star \mathcal{O}_X =\mathcal{O}_S$ and $-K_X$ is $\pi$-ample.
\end{dfn}

It follows from this that $\pi$ is flat, $S$ is nonsingular,
$E=\pi_\ast \oo(-K_X)$ is a rank~$3$ vector bundle on $S$, and every
fibre is a conic. 

\begin{dfn}
The \emph{discriminant} of the conic bundle is the curve
 \begin{equation} \label{eq:discriminant}
 	\Delta=\{ s \in S \ | \ \pi^{-1}(s)  \ \text{is a singular conic} \}  \subset S
      \end{equation}
\end{dfn}
    
It is well-known that $\Delta \subset S$ is a nodal curve, and the set
\[ \Delta_{\mathrm{s}}= \{ s \in S \ | \ \pi^{-1}(s)  \ \text{is a a double line} \} \subset \Delta
\]
is the singular locus of $\Delta$, see Equation~\eqref{eq:minors}. Below we write

\[
  \Delta_{0}\coloneqq \Delta \setminus \Delta_{\mathrm{s}}
\]

\begin{dfn}
  The \emph{ramification data} of the conic bundle is the double cover
  $p \colon \widetilde{\Delta} \to \Delta$, where $\widetilde{\Delta}$
  is the curve parametrizing irreducible components of the singular
  conics over $\Delta$. We denote by $\tau \colon \widetilde{\Delta}
  \to \widetilde{\Delta}$ the involution of the cover.
\end{dfn}

It is well-known~\cite{Bea1, Bel, Bea2} that if $ \pi \colon X \to S $ is a 
conic bundle and $X$ is proper, then $H^3(X, \ZZ)$ is the
$\tau$-noninvariant part of $H^1(\widetilde{\Delta}, \ZZ)$.

If $X$ is not proper, then $H^3(X, \ZZ)$ --- and $H_c^3(X, \ZZ)$ ---
carries a  mixed Hodge structure, and our
Theorem~\ref{thm:sheaf-conic-bdl} states that there is a surjection
$R^2\pi_\star \ZZ_X \to \ZZ_S$ with kernel the subsheaf of $p_\star
\ZZ_{\widetilde{\Delta}}$ of anti-invariants under the
involution. This description of $R^2\pi_\star \ZZ_X$ provides a tool
to study $H^3(X,\ZZ)$ --- and $H_c^3(X, \ZZ)$ --- when $X$ is not
proper.

\smallskip

In Section \ref{sec:proof} we will use
Theorem~\ref{thm:sheaf-conic-bdl} to study $H_c^3(X_\alpha, \ZZ)$,
where $X_\alpha$ is the threefold of Equation~\ref{eq:conic-bdl-1}.
 
\smallskip 

 Let $s \in S$. In a small affine neighbourhood $ s \in U$, one can
 define $X_U\coloneqq \pi^{-1}(U) \subset U \times \PP^2$ by an
 equation of the form:
\begin{equation} \label{eq:chart}
 \sum_{0 \leq i,j \leq 2} a_{ij} x_i x_j
\end{equation} where $a_{ij} \in \CC[U]$. Then $\Delta\cap U=\left(\det(a_{ij})=0\right)$. Moreover, by a change of coordinates, one can rewrite \eqref{eq:chart} as:
\begin{equation} \label{eq:minors}
\begin{split}
 b_0 x_0^2 +b_1 x_1^2 +b_2 x_2^2=0 &\Leftrightarrow \mathrm{rank}(a_{ij})=3 \Leftrightarrow   s \notin \Delta \\
  b_0 x_0^2 +b_1 x_1^2 +c_2 x_2^2=0 &\Leftrightarrow \mathrm{rank}(a_{ij})=2 \Leftrightarrow   s \in \Delta_{0} \\
 b_0 x_0^2 +c_1 x_1^2 +c_2 x_2^2+c_3 x_1x_2=0 & \Leftrightarrow \mathrm{rank}(a_{ij})=1 \Leftrightarrow   s \in \Delta_{\mathrm{s}} 
\end{split}
\end{equation} where $b_i \neq 0$, $c_i=0$, $\mult_s(c_1)=\mult_s(c_2)=1$.

Let $\delta_k$, $k=1,2,3$, be the $2\times2$-minors of the matrix
$(a_{ij})$ obtained by deleting the $k$th row and column.
By \eqref{eq:minors}, on each irreducible component $\Delta_i$ of
$\Delta$, the double cover $p \colon \widetilde{\Delta} \to \Delta$ is
specified by a minor $\delta=\delta_k$ which does not vanish
identically along $\Delta_i$.

\subsection{Proof of Theorem \ref{thm:sheaf-conic-bdl}} \label{sec:conic-proof}

We prove Theorem~\ref{thm:sheaf-conic-bdl}, following the setup and
notation of the previous section.

\begin{lem} Let $\pi \colon X \to S$ be a conic bundle with
  ramification data $p \colon \widetilde{\Delta} \to \Delta$.
The involution $\tau$ on $\widetilde{\Delta}$ induces an involution,
  which for simplicity we still denote by $\tau$, on
  $p_\star \ZZ_{\widetilde{\Delta}}$.  Let
  $\mathbb{W} \subset p_\star \ZZ_{\widetilde{\Delta}}$ be the
  subsheaf of anti-invariants with respect to $\tau$, equivalently
  defined by the short exact sequence:
\begin{equation} \label{eq:W} 
	0 \to \mathbb{W} \to 	p_{\star} \ZZ_{\widetilde{\Delta}} \xrightarrow{p_\star}   \ZZ_{\Delta}  \to 0 
\end{equation}
where $p_{\star} \colon \ZZ_{\widetilde{\Delta}} \to   \ZZ_{\Delta}$ is the trace map.
Consider the diagram: 
\begin{equation}
  \label{eq:diagram}
\begin{tikzcd}  
X_U \ar[d, "\pi_U"] \arrow[r] &   X \ar[d, "\pi"] & \ar[d, "\pi_\Delta"] X_\Delta  \ar[l]\\
 U   \ar[r, "j"]  & \  S &  \Delta \ar[l, "i"']
\end{tikzcd}
\end{equation}
where $U=S\setminus \Delta$,  $X_U=\pi^{-1}(U)$,
$X_\Delta=\pi^{-1}(\Delta)$ and $\pi_U \colon X_U \to U$, $\pi_\Delta
\colon X_\Delta \to \Delta$ are the restrictions.

We have:
\begin{enumerate}[(i)]
\item $R^0\pi_\star \ZZ_X = \ZZ_S$ and $
 R^1\pi_\star \ZZ_X = (0)$;
\item  The sheaf $R^2\pi_\star \ZZ_X$ has stalks:
\begin{equation*}
	(R^2\pi_\star \ZZ_X)_s=H^2(X_s, \ZZ) = \left\{ \begin{array}{ll}
	\ZZ & \textrm{if  $s \in U$ }\\ 
	\ZZ \oplus \ZZ & \textrm{if  $s \in \Delta_{0}$}\\
\ZZ & \textrm{if $s \in \Delta_{s}$}
\end{array} \right.
	\end{equation*}
\item the restriction $j^\star R^2\pi_\star \ZZ_X = R^2 \pi_{U \star} \ZZ_{X_U}$ is isomorphic to  $\ZZ_U(-1)$;
\item the restriction $i^\star R^2\pi_\star \ZZ_X = R^2 \pi_{\Delta
    \star} \ZZ_{X_\Delta}$ is isomorphic to $p_\star
  \ZZ_{\widetilde{\Delta}}(-1)$. 
\end{enumerate}
\end{lem}
\begin{proof}
  To prove (i), note that for all $s \in S$
  $(R^0\pi_\star \ZZ_X)_s = H^0(X_s, \ZZ) = \ZZ$, and
  $\Gamma_S(\pi_\star \ZZ_X) = H^0(X, \ZZ) = \ZZ$. On the other hand,
  for all $s \in S$ $(R^1\pi_\star \ZZ_X)_s = H^1(X_s, \ZZ) =(0)$.

The statement in (ii) is immediate. 

To prove (iii), it is enough to notice that 
$c_1 \oo(1)_{| X_U} \in \Hom(\ZZ_U, R^2 \pi_{U \star} \ZZ_{X_U})=H^0( U,R^2 \pi_{U \star} \ZZ_{X_U})$
is a global section of $R^2 \pi_{U \star} \ZZ_{X_U}$.

To prove (iv), consider the diagram:
\begin{equation*} \label{eq:diagram}
\begin{tikzcd}  
 &  \ar[dr, ] \widetilde{\Delta}_{0}  \times_{\Delta
_0 }X_{\Delta_0} \ar[dl ] & \\
\widetilde{\Delta}_{0} \ar[dr, "p_0"]  & & \ar[dl, "\pi_0"'] X_{\Delta_0}    \\
& \Delta_0 & \end{tikzcd}
\end{equation*}
where $\widetilde{\Delta}_0=p^{-1}(\Delta_0)$, $X_{\Delta_0}=\pi^{-1}(\Delta_0)$ and $p_0 \colon \widetilde{\Delta}_0 \to {\Delta}_0$ and $\pi_0 \colon X_{\Delta_0} \to \Delta_0$ are the restrictions. The fiber product
$
\widetilde{\Delta}_{0}  \times_{\Delta_0} X_{\Delta_0}= Z_1 + Z_2
$ is the sum of two irreducible components. Denoting by $Z$ one of these components, $Z$ induces an isomorphism $z \colon  p_{0  \star} \ZZ_{\widetilde\Delta_0}  \to   R^2 \pi_{0  \star}\ZZ_{X_{\Delta_0}}(1) $.
One concludes by pushing forward both sides of this isomorphism.
\end{proof}

\subsubsection{Proof of Theorem \ref{thm:sheaf-conic-bdl} }
Consider the relative Poincar\'e morphism:
\begin{equation} 
\label{eq:relative}P \colon \ZZ_{X}[2](1) \to D_{X/S} \end{equation}
 where $ D_{X/S}=\pi^{!}\ZZ_S$ is the relative dualizing complex of $X$ over $S$. 
The functor $R^0\pi_\star$ applied to the morphism \eqref{eq:relative} induces the morphism:
\begin{equation} \label{Poincare_R^2p}
P_2 \colon  R^2\pi_\star \ZZ_X(1) \to  R^0\pi_\star  \; \pi^{!} \ZZ_S \end{equation} 
which, on each stalk, 
is the usual Poincar\'e map of $X_s$:
 \begin{equation} \label{eq:poincare-fibre}
P_{2,s} \colon H^2(X_s, \ZZ)(1) \to H_0(X_s, \ZZ)  
\end{equation} 
Note that $R^0\pi_\star \pi^{!} \ZZ_S \simeq \ZZ_S$ 
by local Verdier duality
\[R\mathit{Hom}(R\pi_! \ZZ_X, Z_S) = R \pi_\star R\mathit{Hom}( \ZZ_X,
  \pi^{!} \ZZ_S) \] and that
$P_2$ is a surjective morphism.  In fact, for $s \notin
\Delta_{0}$, the map $P_{2,s}$ is an isomorphism.
For $s \in  \Delta_{0}$, since $X_s$ is the union of two distinct lines, $H^2(X_s, \ZZ)(1) \simeq \ZZ^2$,
and map $P_{2,s}$ sends \[\ZZ^2 \ni (k_1,k_2) \mapsto k_1+k_2  \in \ZZ \] see Figure \ref{fig:disks},  thus 
$\ker(P_{2,s})=\{(k, -k) \in \ZZ^2 \}$. 
It follows that over $\Delta$ the morphism $P_2$ is the same as the morphism $p_\star \ZZ_{\widetilde{\Delta}} \to \ZZ_{\Delta}$, whose kernel is the sheaf $\mathbb{W}$ of anti-invariants with respect to $\tau$. This proves the statement of the theorem. 
\qed

\begin{figure}[ht!]
\definecolor{ududff}{rgb}{0.30196078431372547,0.30196078431372547,1}
\begin{tikzpicture}[line cap=round,line join=round,x=1cm,y=1cm, scale=0.4]
\clip(-16.43494777890499,-5.902417376357542) rectangle (15.904787752645536,5.281544148323221);
\draw [rotate around={167.82826177912065:(-1.929338355669112,2.0003560156581037)},line width=1pt] (-1.929338355669112,2.0003560156581037) ellipse (2.2060049020709944cm and 0.8115841410484816cm);
\draw [rotate around={-156.79065779003835:(-1.6962313123890054,-2.7238622938897668)},line width=1pt] (-1.6962313123890054,-2.7238622938897668) ellipse (1.904608655516325cm and 0.8555315538283338cm);
\draw [line width=1pt] (-2.025,2.1)-- (-1.155,4.8);
\draw [line width=1pt] (-3.225,0.48)-- (-1.8746160858596386,-2.7912328127593793);
\draw [line width=1pt] (-3.075,-1.11)-- (-2.325,1.26);
\draw [line width=1pt,dotted] (-2.325,1.26)-- (-2.025,2.1);
\draw [line width=1pt,dotted] (-1.8746160858596386,-2.7912328127593793)-- (-1.5215294543124387,-3.575869771753156);
\draw [line width=1pt] (-1.5215294543124387,-3.575869771753156)-- (-0.855,-5.13);
\begin{scriptsize}
\draw [fill=ududff] (-2.025,2.1) circle (4pt);
\draw[color=ududff] (-1.2812343856205934,2.0041954777190308) node {$k_1$};
\draw [fill=ududff] (-1.8746160858596386,-2.7912328127593793) circle (4pt);
\draw[color=ududff] (-1.2223866136960602,-2.3000396446964273) node {$k_2$};
\draw[color=black] (-4.2812343856205934,4.0041954777190308) node {$L_1$};
\draw[color=black] (-4.2812343856205934,-4.3000396446964273) node {$L_2$};
\end{scriptsize}
\end{tikzpicture}
\caption{\label{fig:disks} A section of $R^2 \pi_\star \ZZ$ about a degenerate fiber $X_s=L_1 \cup L_2$ looks like the union of two horizontal disks meeting $X_s$ at $k_1 p_1+k_2 p_2$, where $p_i \in L_i$. The map $P_{2,s}$ sends $k_1 p_1+k_2 p_2$ to $(k_2+k_2)p$, where $p$ is the class of the point. }
\end{figure}
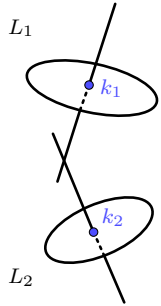




\section{Proof of Theorem \ref{thm:main}}
\label{sec:proof}
In this section we prove Theorem~\ref{thm:main} of the
Introduction. For convenience in what follows we fix $\alpha \neq
\alpha_0$ and we omit all reference to $\alpha$. 

\medskip

\paragraph{\textbf{The partial compactification $Z\subset X$.}}
In Section \ref{sec:conic-structure} we built a partial
compactification $X \subset (\CC^\times)^2 \times \PP^2$ of $Z$.
Lemma \eqref{eq:gr3} below states that
$\gr^W_3 H^3_c(X, \QQ) = \gr^W_3 H^3_c(Z, \QQ)$, thus one might
replace $Z$ with $X$ in Proposition \ref{pro:our-claim}.
The advantage of working with $X$ is that the first projection $\pi \colon X \to (\CC^\times)^2$ is a conic bundle. 

\begin{lem} \label{lem:gr3}
There is an identity of pure Hodge structures: 
\begin{equation} \label{eq:gr3}
	\gr_3^W H^3_c(Z, \ZZ)	= \gr_3^W H^3_c(X, \ZZ)
\end{equation}
\end{lem}

\begin{proof}
Consider the divisor $D=X \setminus Z$. 
There is a long exact sequence of mixed Hodge structures:
\begin{equation}
 \cdots \to H^2_c(D, \ZZ) \to   H^3_c(Z, \ZZ) \to  H^3_c(X, \ZZ) \to  H^3_c(D, \ZZ) \to \cdots
\end{equation}
To prove \eqref{eq:gr3}, we show that $\gr_3^W H^2_c(D, \ZZ)  =\gr_3^W H^3_c(D, \ZZ) =(0)$ . We have that $D=D_0 \cup D_1 \cup D_2$, where $D_i=X \cap (x_i=0)$, $i=0,1,2$, and 
\[ D_0 \cap D_1 =D_0 \cap D_2= \varnothing \quad D_1 \cap D_2 =(u_1^3+u_2^2=0) \subset (\CC^\times)^2
\]
This implies that, for all $i$, $H^i_c(D, \ZZ) =H^i_c(D_0, \ZZ) \oplus H^i_c(D_1 \cup D_2, \ZZ) $.

The surface $D_0$ is given by \[\left( \alpha u_1 x_1^2-x_2^2=0\right) \subset (\CC^\times)^2 \times \PP^1\] Hence $D_0 \simeq (\CC^\times)^2$,  thus $H^2_c(D_0, \ZZ)$ is a pure Hodge structure of weight 0, and  $H^3_c(D_0, \ZZ)$ is a pure Hodge structure of weight $2$. 
It follows that 
\begin{equation}
\gr_3^W H^i_c(D, \ZZ)=\gr_3^W  H^i_c(D_1 \cup D_2, \ZZ) \qquad  i=2,3
\end{equation}

We have a long exact sequence of mixed Hodge structures:
\begin{equation} \label{eq:D1D2} \cdots \to H^{i-1}_c(D_1, \ZZ) \to  H^i_c(D_2\setminus D_1, \ZZ) \to H^i_c(D_1 \cup D_2, \ZZ) \to H^i_c(D_1, \ZZ) \to \cdots
\end{equation}
The surface $D_1$ is the smooth surface given by:
\[ \left((u_1^3u_2+u_2^3)x_0^2+x_0x_2-\frac{1}{\alpha} x_2^2=0 \right) \subset (\CC^\times)^2 \times \PP^1
\] Then a natural compactification of $D_1$ is the smooth degree two del Pezzo surface $\overline{D}_1$ given by 
\[ \left((u_1^3u_2+u_0u_2^3)+u_0^2x_2-\frac{1}{\alpha} x_2^2=0 \right) \subset \PP(1,1,1,2)
\]   Hence $\gr_3^W H^2_c(D_1, \ZZ)  =\gr_3^W H^3_c(D_1, \ZZ) =0$.
The surface $D_2$ is the smooth surface given by:
\[ \left( (u_1^3u_2+u_2^3)x_0^2+u_1u_2 x_0 x_1+u_1 x_1^2=0 \right) \subset (\CC^\times)^2 \times \PP^1
\]
Since $D_2 \setminus D_1$ is nonsingular but noncompact, we have $\gr_3^W   H^2_c(D_1 \setminus  D_2, \ZZ) =(0)$. Then it follows from \eqref{eq:D1D2} that $\gr_3^W   H^2_c(D_1 \cup D_2, \ZZ) =(0)$. On the other hand, since $D_1 \cap D_2$ is a nonsingular and noncompact curve, 
 we have $\gr_3^W   H^3_c(D_2\setminus D_1, \ZZ)= \gr_3^W   H^3_c(D_2, \ZZ)$.
Now, since the projection $\phi \colon D_2 \to (\CC^\times)^2$ to the first factor is  a double (branched) cover of $(\CC^\times)^2$, 
we have that $\gr_3^W   H^3_c(D_2, \ZZ) = \gr_3^W H^3_c((\CC^\times)^2,\phi_\star \ZZ) =(0)$. 
It follows from \eqref{eq:D1D2} that $\gr_3^W   H^3_c(D_1 \cup D_2, \ZZ)=0$. This concludes the proof.
\end{proof}
\smallskip

\paragraph{\textbf{Ramification data of $\pi \colon X \to (\CC^\times)^2$.}} \label{sec:data}
By means of the substitutions
$x_2 \mapsto x_2 +\alpha  x_0/2$, $x_1 \mapsto  x_1 -u_2   x_0/2$,
we rewrite Equation \eqref{eq:conic-bdl-1} as
\begin{equation} \label{eq:conic-bdl-2}
	\left(u_1^3u_2+u_2^3+\frac{\alpha}{4} -\frac{u_1u_2^2}{4}\right)x_0^2 +u_1 x_1^2 -\frac{1}{\alpha}x_2^2=0
\end{equation} 

In matrix notation, we represent \eqref{eq:conic-bdl-2} by the diagonal matrix:
\begin{equation} \label{eq:matrix-conic}
	\left[ \begin{array}{ccc}
u_1^3u_2+u_2^3+\frac{\alpha}{4} -\frac{u_1u_2^2}{4} & 0 & \textcolor{white}{-}0 \\
0 & u_1 & \textcolor{white}{-}0\\ 
0 & 0 & -\frac{1}{\alpha}
\end{array} \right]
\end{equation}
Then the  discriminant of  $\pi \colon X \to (\CC^\times)^2$ is the smooth curve $\Delta \subset (\CC^\times)^2$ defined as:
\begin{equation} \label{eq:Delta}
\Delta= \left(  4 u_1^3u_2+ 4 u_2^3+\alpha -u_1u_2^2=0\right) \subset (\CC^\times)^2
\end{equation} 
The $2:1$ \'etale cover $p \colon \widetilde{\Delta} \to \Delta$ associated to $\pi \colon X \to (\CC^\times)^2$ is specified by the minor $\delta_1=-\frac{1}{\alpha} u_1$ of \eqref{eq:matrix-conic}. Let $u_1, u_2, u_3$ be coordinates on $(\CC^\times)^3$. Then:
\begin{equation}
	\widetilde{\Delta} = \left(\alpha u_3^2-u_1= 4 u_1^3u_2+ 4 u_2^3+\alpha -u_1u_2^2=0 \right) \subset (\CC^\times)^3
\end{equation} with $p$ the projection $(u_1, u_2, u_3) \mapsto (u_1, u_2)$. 
Equivalently, $\widetilde{\Delta}$ is the curve:
\begin{equation} \label{eq:Deltatilde-2}
	\widetilde{\Delta}= \left( 4 \alpha^3 u_1^6u_2+ 4 u_2^3+\alpha -\alpha u_1^2u_2^2=0\right) \subset (\CC^\times)^2
\end{equation}  and $p \colon \widetilde{\Delta} \to \Delta$ is the double cover defined by $u_1, u_2 \mapsto \alpha u_1^2, u_2$. 
Let $M=\ZZ e_1 +\ZZ e_2$ be the character lattice of the torus in \eqref{eq:Delta}
 and let $\widetilde{M}=\ZZ \frac{1}{2}e_1 +\ZZ e_2$ be the character lattice of the torus in \eqref{eq:Deltatilde-2}.
 In Figure \ref{fig:polytope} we draw the Newton polytopes $P, \widetilde{P}$ of the polynomials defining $\Delta$ and $\widetilde{\Delta}$. 

\begin{figure}[ht!]
\definecolor{qqwuqq}{rgb}{0,0.39215686274509803,0}
\definecolor{ccqqqq}{rgb}{0.8,0,0}
\begin{tikzpicture}[line cap=round,line join=round,x=1cm,y=1cm, scale=0.9]
\clip(-7.37,-1.68) rectangle (13.37,3.18);
\draw [line width=1pt] (-2,2)-- (-2,-1);
\draw [line width=1pt] (-2,-1)-- (4,0);
\draw [line width=1pt] (-2,2)-- (4,0);
\draw [->,line width=1.5pt] (-2,-1) -- (-2,0);
\draw [->,line width=1.5pt] (-2,-1) -- (-1,-1);
\begin{scriptsize}
\draw [fill=black] (-2,-1) circle (2.5pt);
\draw[color=black] (-2.19,-1.23) node {$O$};
\draw [fill=black] (4,0) circle (2.5pt);
\draw [fill=black] (-2,2) circle (2.5pt);
\draw [fill=ccqqqq] (0,1) circle (2.5pt);
\draw [fill=ccqqqq] (0,0) circle (2.5pt);
\draw [fill=ccqqqq] (2,0) circle (2.5pt);
\draw [fill=qqwuqq] (-1,0) circle (2.5pt);
\draw [fill=qqwuqq] (1,0) circle (2.5pt);
\draw [fill=qqwuqq] (-1,1) circle (2.5pt);
\draw [fill=qqwuqq] (3,0) circle (2.5pt);
\draw [fill=black] (-2,0) circle (0.5pt);
\draw[color=black] (-2.39,0.10) node {$e_2$};
\draw [fill=black] (-1,-1) circle (0.5pt);
\draw[color=black] (-0.70,-1.09) node {$\frac{1}{2} e_1$};
\end{scriptsize}
\end{tikzpicture}
\caption{\label{fig:polytope} A picture of the Newton polytope $P$ of
  $\Delta$, as well as the Newton polytope $\widetilde{P}$ of
  $\widetilde{\Delta}$.  The two arrows in the picture represent the
  standard lattice basis vectors; the interior lattice points of $P$ are in
  red, and the remaining interior lattice points of $\widetilde{P}$ in
  green. }
\end{figure}
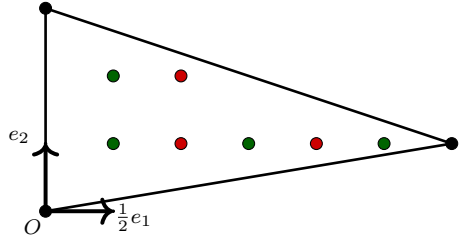


\begin{rem} \label{rem:map-P}
The closure of $\Delta$ in $\PP^2$ is the smooth quartic:
\[ N=\left(4 u_1^3u_2+ 4 u_0 u_2^3+\alpha u_0^4 -u_0u_1u_2^2=0\right) \subset \PP^2
\] where $u_0, u_1, u_2$ are now homogeneous coordinates on $\PP^2$,
and $\Delta=N \setminus \{p_0, \dots, p_4\}$, with $p_0=(0:1:0)$,
$p_i=(u_1= u_2^3+\alpha u_0^3=0)$, $i=1,2,3$, and $p_4=(0:0:1)$, see
Figure \ref{fig:Delta-and-N}. Hence
$\gr^W_1 H^1_c(\Delta, \ZZ)= H^1(N, \ZZ) $ has rank $6$.

The closure of $\widetilde{\Delta}$ in $\PP^2$ is singular at the
point $(0:0:1)$.  Its normalisation is manifestly the genus-$7$ curve
$\widetilde{N} \subset \PP(1,1,3)$ given by
Equation~\eqref{eq:Ntilde}, hence
$\gr^W_1 H^1(\widetilde{\Delta}, \ZZ)= H^1(\widetilde{N}, \ZZ)$ has
rank $14$.
 
We have a commutative diagram:

\begin{equation} \label{eq:diagram2}
\begin{tikzcd}  
\widetilde{\Delta} \ar[d, "p"]   \arrow[hookrightarrow,r]  & \widetilde{N} \ar[d, "P"] \\
 \Delta    \arrow[hookrightarrow,r]   &  N \end{tikzcd}
\end{equation}
where $P \colon \widetilde{N} \to N$ is the quotient by the involution $x_1 \mapsto -x_1$ on $\widetilde{N}$. This is the  map $P$ of Theorem \ref{thm:main}. Note that $P$ has branch divisor $B=p_0+p_1+p_2+p_3$, and that the point $p_0$ is a base point of the linear system $|B|$. 
\end{rem}

\begin{figure}[ht!]
\begin{tikzpicture}[x=0.75pt,y=0.75pt,yscale=-1,xscale=1]

\draw [color={rgb, 255:red, 74; green, 144; blue, 226 }  ,draw opacity=1 ]   (207.5,108) -- (227.5,286) ;
\draw [color={rgb, 255:red, 74; green, 144; blue, 226 }  ,draw opacity=1 ]   (200.5,275) -- (377.5,152) ;
\draw [color={rgb, 255:red, 74; green, 144; blue, 226 }  ,draw opacity=1 ]   (176.5,137) -- (378.5,190) ;
\draw [color={rgb, 255:red, 0; green, 0; blue, 0 }  ,draw opacity=1 ]   (230.5,121) .. controls (233.5,184) and (237.5,185) .. (260.5,149) ;
\draw [color={rgb, 255:red, 0; green, 0; blue, 0 }  ,draw opacity=1 ]   (260.5,149) .. controls (308.33,69) and (259.67,253.67) .. (303.5,260) ;
\draw    (303.5,260) .. controls (350.33,254.33) and (264.33,106.33) .. (471.67,211.67) ;
\draw  [color={rgb, 255:red, 74; green, 144; blue, 226 }  ,draw opacity=1 ][fill={rgb, 255:red, 74; green, 144; blue, 226 }  ,fill opacity=1 ] (335.33,179.83) .. controls (335.33,178.27) and (336.6,177) .. (338.17,177) .. controls (339.73,177) and (341,178.27) .. (341,179.83) .. controls (341,181.4) and (339.73,182.67) .. (338.17,182.67) .. controls (336.6,182.67) and (335.33,181.4) .. (335.33,179.83) -- cycle ;
\draw  [color={rgb, 255:red, 74; green, 144; blue, 226 }  ,draw opacity=1 ][fill={rgb, 255:red, 74; green, 144; blue, 226 }  ,fill opacity=1 ] (252,157.83) .. controls (252,156.27) and (253.27,155) .. (254.83,155) .. controls (256.4,155) and (257.67,156.27) .. (257.67,157.83) .. controls (257.67,159.4) and (256.4,160.67) .. (254.83,160.67) .. controls (253.27,160.67) and (252,159.4) .. (252,157.83) -- cycle ;
\draw  [color={rgb, 255:red, 74; green, 144; blue, 226 }  ,draw opacity=1 ][fill={rgb, 255:red, 74; green, 144; blue, 226 }  ,fill opacity=1 ] (280.67,165.83) .. controls (280.67,164.27) and (281.94,163) .. (283.5,163) .. controls (285.06,163) and (286.33,164.27) .. (286.33,165.83) .. controls (286.33,167.4) and (285.06,168.67) .. (283.5,168.67) .. controls (281.94,168.67) and (280.67,167.4) .. (280.67,165.83) -- cycle ;
\draw  [color={rgb, 255:red, 74; green, 144; blue, 226 }  ,draw opacity=1 ][fill={rgb, 255:red, 74; green, 144; blue, 226 }  ,fill opacity=1 ] (280.17,218.17) .. controls (280.17,216.42) and (281.58,215) .. (283.33,215) .. controls (285.08,215) and (286.5,216.42) .. (286.5,218.17) .. controls (286.5,219.92) and (285.08,221.33) .. (283.33,221.33) .. controls (281.58,221.33) and (280.17,219.92) .. (280.17,218.17) -- cycle ;
\draw  [color={rgb, 255:red, 74; green, 144; blue, 226 }  ,draw opacity=1 ][fill={rgb, 255:red, 74; green, 144; blue, 226 }  ,fill opacity=1 ] (229.33,152) .. controls (229.33,150.34) and (230.68,149) .. (232.33,149) .. controls (233.99,149) and (235.33,150.34) .. (235.33,152) .. controls (235.33,153.66) and (233.99,155) .. (232.33,155) .. controls (230.68,155) and (229.33,153.66) .. (229.33,152) -- cycle ;

\draw (266.67,213.07) node [anchor=north west][inner sep=0.75pt]  [font=\tiny]  {$p_0$};
\draw (217.33,157.73) node [anchor=north west][inner sep=0.75pt]  [font=\tiny]  {$p_1$};
\draw (244.67,143.73) node [anchor=north west][inner sep=0.75pt]  [font=\tiny]  {$p_2$};
\draw (287.33,155.4) node [anchor=north west][inner sep=0.75pt]  [font=\tiny]  {$p_3$};
\draw (332,164.07) node [anchor=north west][inner sep=0.75pt]  [font=\tiny]  {$p_4$};
\end{tikzpicture}
\caption{\label{fig:Delta-and-N} A picture of the smooth quartic $N \subset \PP^2$. The curve $\Delta$ is the complement $N \setminus \{p_0, \dots, p_4\}$. The blue triangle represents the toric boundary of $\PP^2$. 
}
\end{figure}
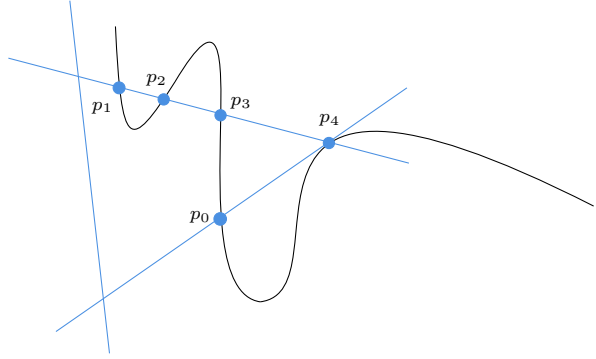

\begin{rem}
  The local system \mbox{$\gr_1^W H^1_c(\Delta_\alpha, \QQ)=H^1(N_\alpha, \QQ)$} is
  the hypergeometric local system associated to 
  $\gamma=(-9,1,3,5)$ --- indeed, $\gamma$ spans the affine relations
  of the monomials in Equations~\eqref{eq:Delta},~\eqref{eq:Deltatilde-2}.
\end{rem}
\smallskip

\begin{lem} \label{lem:leray}
There is an identity of pure Hodge structures:
\begin{equation}
\gr_3^W H^3_c(X, \ZZ) = \gr_3^W H^1_c( (\CC^\times)^2, R^2\pi_\star \ZZ)
\end{equation}
\end{lem}
\begin{proof}
Consider the Leray spectral sequence of  $\pi$ with second page $E_2^{p,q}=H^p_c( (\CC^\times)^2, R^q\pi_\star \ZZ) \Longrightarrow  H_c^{p+q}(X, \ZZ)$. 
Since the second page of the spectral sequence is:
\smallskip

\begin{small}
\[ 
 \begin{array}{ccccc}
   H_c^0((\CC^\times)^2, R^2\pi_\star \ZZ)   &  H_c^1((\CC^\times)^2,  R^2 \pi_\star \ZZ)&  H_c^2((\CC^\times)^2,  R^2 \pi_\star \ZZ) &  H_c^3((\CC^\times)^2,  R^2\pi_\star \ZZ) &   H_c^4((\CC^\times)^2,  R^2\pi_\star \ZZ)\\ (0)  & (0) & (0) & (0) &  (0)
   \\
   H_c^0((\CC^\times)^2, \ZZ)   &  H_c^1((\CC^\times)^2,  \ZZ)&  H_c^2((\CC^\times)^2,\ZZ) &  H_c^3((\CC^\times)^2,  \ZZ) &   H_c^4((\CC^\times)^2,  \ZZ)\end{array} 
\] 
\end{small}
one has that $d_2 \colon E_2^{p,q} \to E_2^{p+2, q-1}$ is zero, thus $E_3^{p,q}= E_2^{p,q}$. 
For reasons of space, the only nonzero $d_3$ 
are:
\[ d_3 \colon H_c^0((\CC^\times)^2,  R^2 \pi_\star \ZZ) \to H_c^3((\CC^\times)^2,  \ZZ) \quad \text{and} \quad d_3 \colon H_c^1((\CC^\times)^2,  R^2 \pi_\star \ZZ) \to H_c^4((\CC^\times)^2,  \ZZ)
\]
 and $d_i=0$ for all $i\geq 4$, thus $E_\infty^{p,q}=E_4^{p,q}$ where
\[\begin{split}
& E_{4}^{0,3}=E_4^{2,1}=0\\
&E_4^{1,2}=\ker\left( d_3 \colon H_c^1((\CC^\times)^2,  R^2 \pi_\star \ZZ) \to H_c^4((\CC^\times)^2,  \ZZ)\right)\\
& E_4^{3,0}=\mathrm{coker} \left(d_3 \colon H_c^0((\CC^\times)^2,  R^2 \pi_\star \ZZ) \to H_c^3((\CC^\times)^2,  \ZZ)\right)
\end{split}
\]
and $d_i=0$ for all $i\geq 4$, vanish, thus $E_4^{p,q}=E_\infty^{p,q}$. 
Then there is a short exact sequence:
\[ 0 \to E_4^{3,0} \to H^3_c(X, \ZZ) \to E_4^{1,2} \to 0\]
Since $H_c^3((\CC^\times)^2,  \ZZ)$ is a pure Hodge structure of weight $2$ and $H_c^4((\CC^\times)^2,  \ZZ)$ is a pure Hodge structure of weight $4$, it follows that
\[ \gr_3^W H^3_c(X, \ZZ) = \gr_3^W E_4^{1,2}   = \gr_3^W   H_c^1((\CC^\times)^2,  R^2 \pi_\star \ZZ)  
\] This proves the statement.
\end{proof}


\begin{lem} \label{lem:implication-our-thm}
As in Section~\ref{sec:conic-proof}, denote by $\mathbb{W}\subset
p_\star \ZZ_{\widetilde{\Delta}}$ the subsheaf of anti-invariants. 

  There is an identity of mixed Hodge structures:
\begin{equation}
H^1_c(\mathbb{W}, \ZZ) = H^1_c( (\CC^\times)^2, R^2\pi_\star \ZZ)(1)	
\end{equation}
\end{lem}
\begin{proof}
By Theorem \ref{thm:sheaf-conic-bdl}, there is a short exact sequence of mixed sheaves:
\begin{equation}
0 \to i_\star \mathbb{W} \to R^2 \pi_\star \ZZ_X(1) \to \ZZ_{(\CC^\times)^2 } \to 0	
\end{equation}
Applying the functor $R^{\bullet} \Gamma_c((\CC^\times)^2, -)$ to this sequence, we obtain the long exact sequence of compactly supported cohomology groups:
\begin{equation}
\cdots \to 	H^0_c( (\CC^\times)^2, \ZZ) \to H^1_c(\mathbb{W}, \ZZ) \to  H^1_c( (\CC^\times)^2, R^2\pi_\star \ZZ)(1)	 \to H^1_c((\CC^\times)^2, \ZZ) \to \cdots 
\end{equation}
The statement follows from the fact that $H^0_c( (\CC^\times)^2, \ZZ) =H^1_c((\CC^\times)^2, \ZZ)=0$.
\end{proof}

\begin{lem} \label{lem:exact-sequence}
There is an an exact sequence:
\begin{equation} \label{eq:ker-Gysin}
0 \to H^1_c(\Delta, \mathbb{W}) \to H^1_c( \widetilde{\Delta}, \ZZ) \xrightarrow{p_\star} 	H^1_c( {\Delta}, \ZZ)
\end{equation} where $p_\star$ is the Gysin morphism.
\end{lem}

\begin{proof}
Apply the functor $R^{\bullet} \Gamma_c(\Delta, -)$ to the the short exact sequence \eqref{eq:W} defining  $\mathbb{W}$ and observe that $H^0_c(\Delta, \ZZ) \simeq H_2(\Delta, \ZZ)=0$.
\end{proof}

\begin{rem} \label{rem:over-Q}
Note that the Gysin morphism $p_\star$ in \eqref{eq:ker-Gysin} is not surjective but has cokernel $\ZZ/2\ZZ$. Over $\QQ$ one has the short exact sequence:
\[ 0 \to H^1_c(\Delta, \mathbb{W}_\QQ) \to H^1_c( \widetilde{\Delta}, \QQ) \xrightarrow{p_\star} 	H^1_c( {\Delta}, \QQ) \to 0
\] where $\mathbb{W}_\QQ = \mathbb{W} \otimes_\ZZ \QQ $.
\end{rem}

\paragraph{\textbf{Proof of Theorem \ref{thm:main}.}}
Applying in sequence Lemma \ref{lem:gr3}, Lemma \ref{lem:leray}, Lemma \ref{lem:implication-our-thm}, one sees that:
\[ \gr_3^W H^3_c(Z, \ZZ)	= \gr_3^W H^3_c(X, \ZZ)= \gr_3^W H^1_c( (\CC^\times)^2, R^2\pi_\star \ZZ)	= \left(\gr_1^W H^1_c(\Delta, \mathbb{W})\right)(-1)
\]
Then the statement follows from  Proposition \ref{pro:our-claim} together with Lemma \ref{lem:exact-sequence}, Remark \ref{rem:map-P}, and Remark \ref{rem:over-Q}. \qed


\appendix 
\section{Proof of Proposition \ref{pro:our-claim}}
\label{sec:appendix}

The proof of Proposition~\ref{pro:our-claim} uses the theory of
GKZ systems: we summarise what we need in the next section.

\subsection{GKZ systems}
\label{sec:gkz}


Let $N$ and $n$ be two positive integers.  Let $A=\{a_1, \dots a_N\}$
be a set of $N$ vectors in $\ZZ^n$, and let
$\overline{A}= \{\overline{a}_1, \dots, \overline{a}_N \} \subset
\ZZ^{n+1}$ where $\overline{a}_i=(1, a_i)$.  Let $\beta$ be a vector
in $\CC^{n+1}$. Assume that the vectors $\overline{a}_i$ generate a
rank-$(n+1)$ sublattice of $\ZZ^{n+1}$.

Let $\CC^N$ be the affine space with coordinates $v_1, \dots,
v_N$. Let $\mathcal{D}_{\CC^N}$ be the sheaf of polynomial linear
partial differential operators on $\CC^N$. We simply write
$\mathcal{D}$ whenever the context allows it.  For all
$i \in \{1, \dots, N\}$ we write $\delta_i=\delta_{v_i}$,
$D_i=v_i \delta_i$.
 
 \begin{dfn}The \emph{GKZ system} with parameters
   $(A,\beta)$ is the differential
   system: \begin{equation} \label{eq:gkz-general} M(A, \beta)=
     \left\{ \begin{array}{ll}
               \sum_{ j=1}^N \overline{a}_j D_i -\beta  \\
               \textcolor{white}{--}\\
 
\prod_{l_i>0} (\delta_i)^{l_i} - \prod_{l_i<0}  (\delta_i)^{-l_i}	  & \text{for} \quad l \in \LL
\end{array} \right.
\end{equation}
where $\LL$ is the lattice of integral relations among the $\overline{a}_i$:
\begin{equation}\LL =\{ 
l=(l_1, \dots, l_n) \in \ZZ^N \colon  \sum_{i=1}^N  l_i \overline{a}_i=0 
\}
\end{equation}
Note that $\LL$ is a lattice of rank $N-n-1$. 
\end{dfn}

\begin{dfn}The {GKZ $\mathcal{D}$-module} with parameters $(A,\beta)$ is the quotient: 
\begin{equation} \mathcal{M}(A, \beta) =\mathcal{D}/(M(A, \beta)) 
\end{equation}
where $ (M(A, \beta))$ denotes the left ideal generated by the partial differential operators of $M(A, \beta)$. 
\end{dfn}
\medskip

Let $\TT=\Spec \CC[\ZZ^n]\simeq (\CC^\times)^n$
and let $u_1, \dots, u_n$ be coordinates on $\TT$. Write $u=(u_1, \dots, u_n)$. 
Every Laurent polynomial $f \in \CC[u_1^{\pm 1}, \dots, u_n^{\pm 1}]$ defines an affine hypersurface 
\begin{equation} Z_{f}=(f=0) \subset \TT
\end{equation} 
Observe that the complement $\TT \setminus Z_{f}$ is isomorphic to the affine hypersurface $Z_{F}=(F=0) \subset \overline{\TT}$ where $\overline{\TT}=\Spec \CC[\ZZ^{n+1}]$ is the $(n+1)$-dimensional torus with coordinates $u_0, u_1, \dots, u_n$, and $F(u_0,u)=x_0 f(u)-1$. 
\smallskip

Now let $A=\{a_1, \dots, a_N\}$ as above. 	 
To every point $v=(v_1, \dots, v_N) \in \CC^N$ one can associate the Laurent polynomial:
\begin{equation} \label{eq:f-v}
f_v(u)=\sum_{j=1}^N v_j u^{a_j}	 \in \CC[u_1^{\pm 1}, \dots, u_n^{\pm 1}]
\end{equation} where $u^m=u_1^{m_1} u_2^{m_2 }\dots u_n^{m_n}$ for $m=(m_1, \dots, m_n) \in \ZZ^n$, the affine hypersurface $Z_{f_v} \subset \TT$, 
and the affine hypersurface $Z_{F_v} \subset \overline{\TT}$.
 We simply write $f, Z_f, F, Z_F$  whenever the the subscript  is clear from the context. 

 The following result relates the $\mathcal{D}$-module
 $\mathcal{M}(A, \underline{0})$ and the Laurent polynomials $f_v$.
 
\begin{thm}{\cite[Theorem 8]{Stienstra}} \label{thm:Stienstra}
Let $A$ be as above. Assume that the following two conditions hold:

\begin{enumerate}[1.]
\item  the vectors $\overline{a}_i=(1,a_i)$ generate  $\ZZ^{n+1}$;

	\item $A=\Delta \cap \ZZ^n$ for some $n$-dimensional lattice polyhedron $\Delta$. 
	 \end{enumerate}  
 
 Let $E_A=E_A(v_1, \dots, v_N)$ be the principal $A$-determinant and
 let $U=\CC^N\setminus \{E_A=0\}$. Then the local system of flat sections
 of $\mathcal{M}(A, \underline{0})_{|U}$ is the variation of
 relative cohomology $H^{n+1}(\overline{\TT} , Z_{F_v}, \QQ)$ on $U$.
\end{thm}
	
\begin{rem}
  Conditions 1 and 2 are necessary for Theorem
  \ref{thm:Stienstra} to hold, see \cite[Part IIB]{Stienstra}. To
  prove Proposition \ref{pro:our-claim}, we will need the conclusion
  of Theorem~\ref{thm:Stienstra} to hold for a GKZ system that does not satisfy
  Condition 2.  \end{rem}

\begin{rem}
  \label{rem:sequences}
  If the Laurent polynomial $f$ is $\Delta$-regular~\cite[Definition
  3.3]{Batyrev} there is an exact sequence of mixed Hodge structures:
  \begin{equation} \label{eq:exact-sequence-rel} 0 \to
    H^n(\overline{\TT}, \QQ) \to H^n(Z_F, \QQ) \to
    H^{n+1}(\overline{\TT} , Z_F, \QQ) \to H^{n+1}(\overline{\TT},
    \QQ) \to 0
\end{equation}
see \cite[Equation (55)]{Stienstra}. 
This sequence  
splits into the two short exact sequences:
		\begin{equation} \label{eq:def-primitive}
		0 \to H^n(\overline{\TT}, \QQ) \to H^n(Z_F, \QQ)  \to PH^n(Z_F, \QQ) \to 0
	\end{equation}
	\begin{equation} \label{eq:relevant-ses}
		0 \to PH^n(Z_F, \QQ) \to H^{n+1}(\overline{\TT} , Z_F, \QQ) \to H^{n+1}(\overline{\TT}, \QQ)  \to 0
\end{equation}
where the symbol $PH^{\bullet}$ denotes primitive cohomology. 

One has that $H^n( \overline{\TT}, \QQ)= \QQ^{n+1}(-n)$ and
\mbox{$H^{n+1} (\overline{\TT}, \QQ) \simeq \QQ(-(n+1))$}.  By
\cite[Proposition 5.2]{Batyrev},
$\dim H^n(Z_F, \QQ)=\mathrm{vol}(\Delta)+n$, thus
$\dim PH^n(Z_F, \QQ)=\mathrm{vol}(\Delta)-1$ and
$\dim H^{n+1}(\overline{\TT} , Z_F, \QQ)=\mathrm{vol}(\Delta)$.  Note
that, by \eqref{eq:relevant-ses},
\[ W_i PH^n(Z_F, \QQ) =W_i H^{n+1}(\overline{\TT}, Z_F, \QQ) \qquad \text{for all} \ i \leq 2n+1
\]  
Moreover, the residue map
$\mathrm{Res} \colon H^{n}(Z_F, \QQ) \twoheadrightarrow H^{n-1}(Z_f,
\QQ)(-1)$ induces an isomorphism
$ PH^n(Z_F, \QQ) \simeq PH^{n-1}(Z_f, \QQ)(-1) $, see \cite[Section
5]{Batyrev}.

\end{rem}


\subsection{Proof of Proposition \ref{pro:our-claim}}
 We follow the same setup and notation of Section \ref{sec:gkz}.
 
\subsubsection*{Preliminary lemmas}
Consider the order-$19$ reducible hypergeometric operator
$\widetilde{H}$ on $\CC^\times$:\footnote{The hypergeometric operator
  $\widetilde{H}$ can be computed from the coefficients of the
  function $F$ via the two-term recursion
  relation
  , see \cite[Section 5.3]{NK}.}
\begin{equation} \label{eq:reducible} \widetilde{H}= \alpha_0
  D\left(D-\frac{1}{2}\right)
  \prod_{i=0}^{2}\left(D-\frac{i}{3}\right)
  \prod_{i=0}^{4}\left(D-\frac{i}{5}\right)
  \prod_{i=0}^{8}\left(D-\frac{i}{9}\right) - \alpha
  (D+1)\prod_{i=1}^{18}\left(D+\frac{i}{18}\right)
\end{equation}  
where $D=\alpha d/d \alpha$ as in Section \ref{sec:initial-set-up}.
Since $\alpha(D+1+\delta)=(D+\delta)\alpha$ for all $\delta \in \CC$,
we have $\widetilde{H}=G \cdot H$, where $H$ is the operator in
Equation \eqref{eq:irreducible} and
\[G= D\left(D-\frac{1}{2}\right)
  \prod_{i=0}^{8}\left(D-\frac{i}{9}\right)
\]

\smallskip

Let $\Delta$ be the convex hull of the vectors $m_i$, $i=1, \dots, 6$
in Equation \eqref{eq:our-choice-mi}. The lattice points of $\Delta$
are given by the $m_i$ and the vector $(0,1,0,1)$. Let $a_i=m_i$,
$i=1, \dots, 6$, and write $A^\prime=\{a_1, \dots, a_6\}$.

\begin{lem} \label{lem:Aprime-and-Hred} Let $k_i$ be integers such
  that $\sum_{i=1}^6 k_i \gamma_i=1$ and consider the morphism
  \[
    \CC^\times \overset{j}{\hookrightarrow} \CC^6
\quad 
\text{given by}
\quad
 v_i \mapsto (-\alpha)^{k_i}
  \]


  Then there is an isomorphism of
  $\mathcal{D}=\mathcal{D}_{\CC^\times}$-modules:
\begin{equation}
\mathcal{D}/ \mathcal{D} \widetilde{H} = j^\star \mathcal{M}(A^\prime, \underline{0})
\end{equation}
\end{lem}

\begin{proof}[Sketch of Proof] The lattice $\LL$ of integral relations among $\overline{a}_1, \dots, \overline{a}_6$ is generated by $\gamma=(-18,-1,2,3,5,9)$. Thus
the GKZ system $M(A^\prime, \underline{0})$ is the system of partial differential equations:
\begin{equation} \label{eq:Mprime}
M(A^\prime, \underline{0})=\left\{ \begin{array}{ll}
 \sum_{i=1}^6 D_i=0\\
 	 D_1+3D_4+D_6=0\\
 	D_1+D_4+3D_5=0 \\
 	D_1+2D_6=0 \\
  2D_2+D_3=0 \\
 	\delta_3^2  \delta_4^3  \delta_5^5  \delta_6^9 - \delta_1^{18} \delta_2=0
\end{array} \right.
\end{equation}
The key point is the following. A function $\Phi (v_1, \dots, v_6)$ satisfies
the first five equations of the system if and only if $\Phi(v_1, \dots, v_6)= G(\prod_{i=1}^6
v_i^{\gamma_i})$ for some function of one variable $G$; then the last equation of the system is satisfied if and only
if $G$ satisfies the hypergeometric ODE:
\[
2D(2D-1)  3D (3D-1)(3D-2) \; 5D \cdots (5D-4)	\; 9D \cdots
(9D-8) - (- \alpha) \cdot  18D \cdots (18D+17)D=0 \]
It follows that $\Phi(j(\alpha))=G(-\alpha)$ is solution to the hypergeometric operator:
\[
L=2D(2D-1)  3D (3D-1)(3D-2) \; 5D \cdots (5D-4)	\; 9D \cdots
(9D-8) - \alpha\cdot  18D \cdots (18D+17)D
\]
and by \cite[Lemma 3.3]{NK} the $\mathcal{D}$-module $\mathcal{D}/\mathcal{D} L$
 is isomorphic to the $\mathcal{D}$-module
$\mathcal{D}/\mathcal{D} \widetilde{H}$. 
\end{proof}

\begin{rem}
  This is why GKZ $\mathcal{D}$-modules are called `hypergeometric'.
\end{rem}

\smallskip

To $(v_1, \dots, v_6) \in \CC^6$ associate the Laurent polynomial 
$f_{(v_1, \dots, v_6)}= \sum_{j=1}^6 v_j u^{a_j}$ as in \eqref{eq:f-v}, 
the hypersurface $Z_{f_{(v_1, \dots, v_6)}} \subset \TT$,  and the hypersurface   $Z_{F_{(v_1, \dots, v_6)}} \subset \overline{\TT}$.

\begin{lem} \label{lem:Stienstra-for-Aprime}
  Let $U^\prime = \CC^6 \setminus \{E_A=0\}$.
  The local system of flat sections of $\mathcal{M}(A^\prime,
    \underline{0})_{|U^\prime}$ is  isomorphic to $H^5(\overline{\TT},
    Z_{F_{(v_1, \dots, v_6)}}, \QQ)$ on $U^\prime$. 
  \end{lem}
  
\begin{proof}
Let $a_7=(0,1,0,1)$ and write $A=\{a_1, \dots, a_7\}$. 
Since the GKZ system $M(A, \underline{0})$ satisfies the two
conditions of Theorem \ref{thm:Stienstra}, letting $U=\CC^7\setminus
\{E_A=0\}$, the local system of flat sections of $\mathcal{M}(A, \underline{0})_{|U}$ is
isomorphic to the variation
$H^5(\overline{\TT}, Z_{F_{(v_1, \dots, v_7)}},
\QQ)$ on $U$.
Let \[ \iota \colon \CC^6 = (v_7=0) \hookrightarrow \CC^7
\] be the inclusion. 
It is clear that the restriction of the variation
$H^5(\overline{\TT}, Z_{F_{(v_1, \dots, v_7)}}, \QQ)$ to the
hyperplane \mbox{$(v_7=0)$} is the variation
$H^5(\overline{\TT}, Z_{F_{(v_1, \dots, v_6)}}, \QQ)$. Then it is
enough to show that the restriction
$\iota^\star (\mathcal{M}(A, \underline{0}))= \mathcal{M}(A^\prime,
\underline{0})$. This follows from the proof of the main result
of~\cite[Theorem 2.2]{FF} and it can be shown computationally with the
\texttt{Dmodules.m2} package~\cite{DmodulesSource} of
\texttt{Macaulay2}~\cite{M2}.
\end{proof}

\subsubsection*{Proof of Proposition \ref{pro:our-claim}.}

Lemma \ref{lem:Aprime-and-Hred} and Lemma \ref{lem:Stienstra-for-Aprime} imply that 
the local system of flat sections of  $\mathcal{D}/  \mathcal{D} \widetilde{H}$	
is the variation  $H^5(\overline{T}, Z_{F_{(v_1, \dots, v_6)}}, \QQ)$, 
where
\begin{equation} \label{eq:vi-and-alpha}
	 v_i=(-\alpha)^{k_i} \quad i=1, \dots, 6 \quad \text{and} \quad \sum_{i=1}^6 k_i \gamma_i=1
\end{equation}
Note that the hypersurface $Z_{f_{(v_1, \dots, v_6)} }\subset \TT$, with the $v_i$ as in \eqref{eq:vi-and-alpha} is (up to isomorphism) the hypersurface $Z_\alpha \subset \TT$ of Section \ref{sec:initial-set-up}. In what follows we omit the subscript $-_{(v_1, \dots, v_6)}$ from our notation.

Denote by $\widetilde{\HH}$ the local system of solutions to
$\widetilde{H}$. Since $\widetilde{H}
\in \mathcal{D}H$, there is an inclusion of local systems $\HH \hookrightarrow \widetilde{\HH}$. 
For all $\alpha \neq \alpha_0$ in $\CC^\times$ we have a short exact sequence:
\begin{equation*}
0 \to PH^4(Z_F, \QQ) \to H^{5}(\overline{\TT} , Z_F, \QQ) \to H^{5}(\overline{\TT}, \QQ)  \to 0	
\end{equation*}
see Remark \ref{rem:sequences}. 
Its dual is:
\begin{equation*}
0 \to H^5_c(\TT, \QQ) \to \mathrm{Hom}( H^{5}(\overline{\TT} , Z_F), \QQ) \to 	 PH_c^4(Z_F, \QQ)(-1) \to 0
\end{equation*}
where $PH^{\bullet}_c$ denotes primitive cohomology with compact support.

One has $H^5_c(\TT) \simeq \QQ$. Moreover, since $PH^4(Z_F, \QQ) \simeq PH^3(Z_f, \QQ)(-1)$, we have $PH_c^4(Z_F, \QQ)(-1) \simeq PH_c^3(Z_f, \QQ)$.  It follows that, for all $m \geq 1$, one has
\[  W_m \mathrm{Hom}( H^{5}(\overline{\TT} , Z_F), \QQ) \simeq W_m PH_c^3(Z_f, \QQ)
\] 
By Batyrev--Borisov's formula \cite{MR1408560},
the Hodge--Deligne numbers $h^{p,q}$ of $PH_c^3(Z_f, \QQ)$ are given by: 
\[
\begin{matrix}
\begin{array}{cccc}
0 &  &   &\\
0 & 4 &   & \\
1 & 7 & 4   &\\
1 & 1 & 0  & 0
\end{array} \end{matrix}
\]
where we  put $h^{p,q}$ in position $(p,q) \in \ZZ^2$, with $(0,0)$ lying in the bottom-left corner. 
Then, since $\HH$ is irreducible,  the weight-one variation it supports must be the variation $\gr^W_3 PH_c^3(Z_f, \QQ)(1)$ over $\CC^\times \setminus \{\alpha_0\}$.  Note that  $\gr^W_3 PH_c^3(Z_f, \QQ)= \gr^W_3 H_c^3(Z_f, \QQ)$. This concludes the proof. 

\qed


\bibliographystyle{amsalpha} 
\bibliography{bib-acgg}

\providecommand{\bysame}{\leavevmode\hbox to3em{\hrulefill}\thinspace}
\providecommand{\MR}{\relax\ifhmode\unskip\space\fi MR }
\providecommand{\MRhref}[2]{%
  \href{http://www.ams.org/mathscinet-getitem?mr=#1}{#2}
}
\providecommand{\href}[2]{#2}
\begin{thebibliography}{{Rod}19}

\bibitem[Bat93]{Batyrev}
Victor~V. Batyrev, \emph{Variations of the mixed {H}odge structure of affine
  hypersurfaces in algebraic tori}, Duke Math. J. \textbf{69} (1993), no.~2,
  349--409. \MR{1203231}

\bibitem[BB96]{MR1408560}
Victor~V. Batyrev and Lev~A. Borisov, \emph{Mirror duality and string-theoretic
  {H}odge numbers}, Invent. Math. \textbf{126} (1996), no.~1, 183--203.
  \MR{1408560}

\bibitem[BCM15]{BCM}
Frits Beukers, Henri Cohen, and Anton Mellit, \emph{Finite hypergeometric
  functions}, Pure Appl. Math. Q. \textbf{11} (2015), no.~4, 559--589.
  \MR{3613122}

\bibitem[Bea77]{Bea1}
Arnaud Beauville, \emph{Vari\'et\'es de {P}rym et jacobiennes
  interm\'ediaires}, Scientific annals of the \'Ecole Normale Sup\'erieure
  (1977), no.~3, 309--391 (fr).

\bibitem[Bea89]{Bea2}
\bysame, \emph{Prym varieties: a survey}, Theta functions---{B}owdoin 1987,
  {P}art 1 ({B}runswick, {ME}, 1987), Proc. Sympos. Pure Math., vol.~49, Amer.
  Math. Soc., Providence, RI, 1989, pp.~607--620. \MR{1013156}

\bibitem[Bel85]{Bel}
Mauro~C. Beltrametti, \emph{On the {C}how group and the intermediate jacobian
  of a conic bundle}, Annali di Matematica Pura ed Applicata \textbf{141}
  (1985), 331--351.

\bibitem[BH89]{BH}
Frits Beukers and Gerrit~J. Heckman, \emph{Monodromy for the hypergeometric
  function {$_nF_{n-1}$}}, Invent. Math. \textbf{95} (1989), no.~2, 325--354.
  \MR{974906}

\bibitem[CG11]{CG}
Alessio Corti and Vasily~V. Golyshev, \emph{Hypergeometric equations and
  weighted projective spaces}, Sci. China Math. \textbf{54} (2011), no.~8,
  1577--1590. \MR{2824960}

\bibitem[CG21]{ACGG}
Alessio Corti and Giulia Gugiatti, \emph{Hyperelliptic integrals and mirrors of
  the {J}ohnson-{K}oll\'{a}r del {P}ezzo surfaces}, Trans. Amer. Math. Soc.
  \textbf{374} (2021), no.~12, 8603--8637. \MR{4337923}

\bibitem[FFW11]{FF}
Mar\'{\i}a-Cruz Fern\'{a}ndez-Fern\'{a}ndez and Uli Walther, \emph{Restriction
  of hypergeometric {$\mathcal{ D}$}-modules with respect to coordinate
  subspaces}, Proc. Amer. Math. Soc. \textbf{139} (2011), no.~9, 3175--3180.
  \MR{2811272}

\bibitem[GS]{M2}
Daniel~R. Grayson and Michael~E. Stillman, \emph{Macaulay2, a software system
  for research in algebraic geometry}, Available at
  \url{http://www.math.uiuc.edu/Macaulay2/}.

\bibitem[JK01]{JK}
Jennifer~M. Johnson and J\'{a}nos Koll\'{a}r, \emph{K\"{a}hler--{E}instein
  metrics on log del {P}ezzo surfaces in weighted projective 3-spaces}, Ann.
  Inst. Fourier (Grenoble) \textbf{51} (2001), no.~1, 69--79. \MR{1821068}

\bibitem[Kat90]{NK}
Nicholas~M. Katz, \emph{Exponential sums and differential equations}, Annals of
  Mathematics Studies, vol. 124, Princeton University Press, Princeton, NJ,
  1990. \MR{1081536}

\bibitem[LT]{DmodulesSource}
Anton Leykin and Harrison Tsai, \emph{{Dmodules: A \emph{Macaulay2} package.
  Version~1.4.0.1}}, A \emph{Macaulay2} package available at
  \url{https://github.com/Macaulay2/M2/tree/master/M2/Macaulay2/packages}.

\bibitem[NR95]{MR1360615}
Donihakkalu~Shankar Nagaraj and Sundararaman Ramanan, \emph{Polarisations of
  type {$(1,2,\cdots,2)$} on abelian varieties}, Duke Math. J. \textbf{80}
  (1995), no.~1, 157--194. \MR{1360615}

\bibitem[{Pro}18]{YP}
Yuri~G. {Prokhorov}, \emph{The rationality problem for conic bundles}, Russian
  Mathematical Surveys \textbf{73} (2018), no.~3, 375.

\bibitem[{Rod}19]{FRV-DUKE}
Fernando {Rodriguez Villegas}, \emph{Mixed {H}odge numbers and factorial
  ratios}, arXiv e-prints (2019).

\bibitem[RRV22]{R-V}
David~P. Roberts and Fernando Rodriguez~Villegas, \emph{Hypergeometric
  motives}, Notices Amer. Math. Soc. \textbf{69} (2022), no.~6, 914–929.

\bibitem[Sti98]{Stienstra}
Jan Stienstra, \emph{Resonant hypergeometric systems and mirror symmetry},
  Integrable systems and algebraic geometry ({K}obe/{K}yoto, 1997), World Sci.
  Publ., River Edge, NJ, 1998, pp.~412--452. \MR{1672077}

\end{thebibliography}
\end{document}